\documentclass[runningheads]{llncs}
\usepackage[T1]{fontenc}
\usepackage[normalem]{ulem} 
  
\usepackage{graphicx}
\usepackage{tikz}

\usepackage{amsmath}
\usepackage{amsfonts}
\usepackage{amssymb}\usepackage{makecell}
\usepackage[]{mdframed}

\begin{document}
\title{Permutation clones that preserve relations}

\newcommand{\N}{\mathbb N}
\newcommand{\Z}{\mathbb Z}
\newcommand{\Q}{\mathbb Q}
\newcommand{\B}{\mathbb B}
\newcommand{\A}{\mathbb A}
\newcommand{\C}{\mathbb C}
\newcommand{\OO}{{\cal O}}
\newcommand{\RR}{{\cal R}}
\newcommand{\Pow}{{\cal P}}

\newcommand{\Fix}{\operatorname{Fix}}
\newcommand{\Pol}{\operatorname{Pol}}
\newcommand{\PPol}{\operatorname{PPol}}
\newcommand{\Inv}{\operatorname{Inv}}
\newcommand{\Cons}{\operatorname{Cons}}
\newcommand{\Aff}{\operatorname{Aff}}
\newcommand{\Deg}{\operatorname{Deg}}
\newcommand{\Par}{\operatorname{Par}}
\newcommand{\Aut}{\operatorname{Aut}}
\newcommand{\Inn}{\operatorname{Inn}}
\newcommand{\Out}{\operatorname{Out}}
\newcommand{\bigboxplus}{\scalebox{1.6}{$\boxplus$}}

%
%
\author{Tim Boykett\inst{1}\orcidID{0000-0002-3003-0927} }
\authorrunning{T. Boykett}
%
\institute{Institute for Algebra, Johannes-Kepler University, Linz, Austria \\
Time's Up Research, Linz, Austria 
\email{tim@timesup.org}\\
\url{http://timesup.org} \\
University for Applied Arts, Vienna}
\maketitle              
\begin{abstract}
Permutation clones generalise permutation groups and clone theory.
We investigate permutation clones defined by relations, or equivalently,
the automorphism groups of powers of relations.
We find many structural results on the lattice of all relationally defined permutation clones on a finite set.
We find all relationally defined permutation clones on two element set.

We show that all maximal borrow closed permutation clones 
are either relationally defined or cancellatively defined.

Permutation clones generalise clones to permutations of $A^n$. Emil Je\v{r}\'{a}bek  found the dual structure to be weight mappings $A^k\rightarrow M$ to a commutative monoid, generalising relations. We investigate the case when the dual object is precisely a relation, equivalently, that $M={\mathbb B}$, calling these relationally defined permutation clones. 
 We determine the number of relationally defined permutation clones on two elements (13).
 We note  that many infinite classes of  clones  collapse when looked at as permutation clones.



\end{abstract}
%
%
%

%
%

 \section{Introduction}

Given a finite set of logic signals, the reversible gates for these signals
are the basis upon which a reversible computer can be designed and built.
The number of logic signals, whether binary logic, ternary logic or some higher arity logic,
will affect which gates are possible.
In this paper we deepen our understanding of these gates and the ways that they can be combined to engineer reversible computer systems.
In particular we will borrow strongly from clone theory, itself developed as an abstract
theory of circuits.

For a given finite set of logic signals $A$, a $k$-ary gate is a bijection from $A^k$ to itself, that is, a permutation of $A^k$.  Thus we investigate collections of
permutations of $A^k$ for all $k$, closed under natural operations of parallel and serial composition.
Taking the lead from \cite{jerabek18} we call closed systems of bijections 
\emph{permutation clones}.
We  consider ancilla and borrow closure, where an extra input and output is allowed; an \emph{ancilla} logic state is provided and returned in a particular state,  whereas a \emph{borrowed} logic state is provided and returned in an arbitrary state.


In previous papers, Aaronson, Grier and Schaeffer \cite{aaronsonetal15} have determined all ancilla closed gates on a set of order 2, and the author, together with  Kari and  Salo, 
has investigated generating sets and other topics \cite{b19,bks17}.
More related work is outlined in the next section.

In this paper, we examine permutation clones that are naturally related to clones.
We  consider permutation clones that are defined by their components maps as clones.
Clones are defined by the relations that they respect using a Galois connection.
Through these results, we gain some insight into the general structure of
reversible gate systems.

We start Section 2 by introducing the background properties of permutation clones and the duality results from Je\v{r}\'{a}bek.
We then investigate the properties of relationally defined permutation clones 
and investigate the various maximal classes.
We conclude with some open questions and ideas for further work.

\section{Background}

We start by introducing some clone theory and permutation group theory.
In particular we are interested in co-clones or relational clones as well as automorphism groups of relations.
We will then bring these together as permutation clones, and introduce weight preservation
as the appropriate dual structure.

Let $A$ be a finite set.
Define  ${\cal O}_A^n$ to be the set of all $n$-ary mappings $f:A^n\rightarrow A$.
Let $\OO(A) = \{f:A^n\rightarrow A \mid n \in \N\}= \cup_n{\cal O}_A^n$ be the collections of multivariate mappings of $A$ to itself.
Such mappings can be transformed by 
\begin{enumerate}
\item permuting variables: if $f:A^n\rightarrow A$ and $\sigma\in S_n$ is a permutation, define $g:A^n\rightarrow A$ by $g(x_1,\dots,x_n)= f(x_{\sigma^{-1}(1)},\dots, x_{\sigma^{-1}(n)})$.
\item identifying variables:  if $f:A^n\rightarrow A$ and $n > 1$ , define $g:A^{n-1}\rightarrow A$ by $g(x_1,\dots,x_{n-1})= f(x_1,\dots, x_{n-1},x_{n-1})$.
\item composition of functions:  if $f:A^n\rightarrow A$ and  $g:A^{m}\rightarrow A$, define $h:A^{m+n-1}\rightarrow A$ by $h(x_1,\dots,x_{m+n-1})= f(g(x_1,\dots, x_{n}),x_{n+1},\dots,x_{n+m-1})$.
\item inserting dummy variables:  if $f:A^n\rightarrow A$  , define $g:A^{n+1}\rightarrow A$ by $g(x_1,\dots,x_{n+1})= f(x_1,\dots, x_{n})$.
\end{enumerate}
The projection mapping $\pi_i^n:A^n\rightarrow A$ is defined by $\pi_i^n(x_1,\dots,x_n)=x_i$.
A collection of mappings $C \subseteq \OO(A) $ is called a \emph{clone} if it contains all projections
and is closed under the transformation operations above.
A set that is closed under the operations but does not include the projections is called an \emph{iterative algebra}.

Let $R\subseteq A^k$ be a $k$-ary relation on $A$.
Let $f:A^n\rightarrow A$ be a $n$-ary map on $A$.
We can extend $f$ to act on $A^k$ componentwise, so $f:(A^k)^n \rightarrow A^k$ in a natural fashion.
If $f$ maps $R$ to itself, then we say that $f$ \emph{respects} the relation $R$, written $f \triangleright R$.
We call this a \emph{polymorphism} of $R$.
This gives rise to a Galois connection between maps and relations on $A$, a duality.

\begin{definition}[\protect{\cite[Section 2.4]{lau_2006}}]
A \emph{relational clone} is a collection of relations closed under the following operations:
\begin{enumerate}
\item permutation of entries
\item projection; ignoring one entry
\item cartesian products
\item intersection
\item includes the equality relation $\{(a,a,b) \mid a,b \in A\}$
\end{enumerate}
\end{definition}

\begin{lemma}
 Let $C$ be a  clone. Then there exists a relational clone $\RR$ such that for all $f \in \OO(A)$,  
 $f \in C$ iff $f  \triangleright R$ for all $R \in \RR$.
\end{lemma}

The collection of maps that respect a relation $R$ is written $\Pol(R)$, the \emph{polymorphisms} of $R$.
All clones are of the form $\bigcap_i \Pol(R_i)$ for some (possibly infinite) set of relations $R_i$.
We call relations the \emph{dual structure} of clones.
Clone theory has been well developed for several decades, see e.g.\ \cite{kerkhoffetal,lau_2006} for an overview.

One of the classical results in clone theory is the determination of the maximal clones.

\begin{theorem}[Rosenberg]
 Let $A$ be a finite set. Then the maximal subclones of ${\mathcal O}(A)$ are one of the following.
 \begin{enumerate}
 \item monotone mappings, that is respecting a bounded partial order on $A$
 \item respecting a graph of prime length loops
 \item respecting a nontrivial equivalence relation
 \item affine mappings for a prime $p$: that is, respecting the relation $\{(a,b,c,d) \mid a+b = c+d\}$ where $(A,+)$ is an elementary abelian group
 \item respecting a central relation
 \item respecting a h-generated relation
 \end{enumerate}
\end{theorem}

Let $A$ be a finite set. 
We will assume a basic understanding of permutation group theory, but clarify some terminology here, for details see e.g.\ \cite{mortimerdixon}.
$Sym(A)=S_A$ is the set of permutations or bijections of $A$.
If $A=\{1,\dots,n\}$ we will often write $S_n$.
We write permutations in cycle notation and act from the right,
so we write the action of a permutation $g\in G \leq Sym(A)$ on an element $a\in A$ as $a^g$.

Given a set $A$ and a $K$-ary relation $R \subseteq A^k$, the collection $\Aut(R) \leq S_A$
of automorphisms of that relation $R$ is a group.
For example the relation defining the edges of a square on $A=\{1,2,3,4\}$ with
\begin{equation}
R=\{ (1,2),(2,3),(3,4),(4,1),(2,1),(3,2), (4,3),(1,4)\} \subseteq A^2
\end{equation}
 has $\Aut(R)$
the dihedral group with 8 elements.
We can speak of $(A,R)$ as a \emph{relational structure} in the same way that an operation $f:A^n\rightarrow A$
gives $(A,f)$ an algebraic structure.


Let $G$ be a group of permutations of a set $A$.
Let $n \in \N$.
The \emph{wreath product} $G wr S_n$ is a group of permutations acting on $A^n$.
The elements of $G wr S_n$ are $\{(g_1,\dots,g_n,\alpha) \mid g_i \in G,\,\alpha \in S_n\}$ with action defined by:
for $(a_1,\dots,a_n)\in A^n$,  
\begin{eqnarray}
(a_1,\dots,a_n)^{(g_1,\dots,g_n,\alpha)} 
= (a_{\alpha^{-1}(1)}^{g_1},\dots,a_{\alpha^{-1}(n)}^{g_n})
\end{eqnarray}


Let $B_n(A) = Sym(A^n)$ and $B(A) = \bigcup_{n\in \N} B_n(A)$. We call $B_n(A)$ the set of \emph{$n$-ary reversible gates} on $A$,
$B(A)$ the set of \emph{reversible gates}. 
For $\alpha \in S_n$, let $\pi_\alpha \in B_n(A)$ be defined by $\pi_\alpha(x_1,\dots,x_n) = (x_{\alpha^{-1}(1)},\dots,x_{\alpha^{-1}(n)})$.
We call this a \emph{wire permutation}.
Let $\Pi=\{\pi_\alpha \mid \alpha \in S_n,\,n\in\N\}$.
In the case that $\alpha$ is the identity, we write $i_n=\pi_\alpha$, the $n$-ary identity.
Let $f\in B_n(A)$, $g\in B_m(A)$. 
Note that we can write $f$ as $(f_1,\dots,f_n)$ with each $f_i:A^n\rightarrow A$, calling $f_i$ the \emph{components} of $f$.
Define the \emph{parallel composition} as
$f\oplus g \in B_{n+m}(A)$ with 
\begin{align*}
 (f\oplus g)(x_1,\dots,x_{n+m}) &= (f_1(x_1,\dots,x_n),\dots,f_n(x_1,\dots,x_n),\\
 &\hspace{10mm}g_1(x_{n+1},\dots,x_{n+m}),\dots
  \dots,g_m(x_{n+1},\dots,x_{n+m})).
\end{align*}
For $f,g\in B_n(A)$ we can compose $f\bullet g$ in $Sym(A^n)$. If $f,g$ have distinct arities we ``pad'' them with identity.
Let $f\in B_n(A)$ and $g\in B_m(A)$, $n< m$, and define
 \begin{eqnarray*}
f\bullet g = (f \oplus i_{m-n}) \bullet g
\end{eqnarray*}
Let $f\in B_n(A)$ and $g\in B_m(A)$, $n> m$, and define
\begin{eqnarray*}
f\bullet g =f \bullet  (g \oplus i_{n-m})
\end{eqnarray*}

We can thus serially compose all elements of $B(A)$.

We call a subset $C \subseteq B(A)$ that includes $\Pi$ 
and is closed under $\oplus$
and $\bullet$ a \emph{permutation clone} \cite{jerabek18}.
 These have also been investigated with ideas from category theory \cite{lafont93} and as \emph{memoryless computation} \cite{mgpc14,mgsr15}.
If we do not insist upon the inclusion of $\Pi$, then we have \emph{reversible iterative algebras}  \cite{bks17}.
If $C,D$ are permutation clones, with $C\subseteq D$, then we call $C$ a \emph{sub permutation clone} of $D$.
For a set $F \subseteq B(A)$ we write $\langle F \rangle$ as the smallest permutation clone that includes $F$, the permutation clone \emph{generated} by $F$.

Let $C$ be a permutation clone. 
We write $C^{[n]} = C \cap B_n(A)$ for the elements of $C$ of arity $n$.
We will occasionally write $(a_1,\dots,a_n)\in A^n$ as $a_1a_2\dots a_n$ for brevity.
We observe that  $C$ is a $\N$-indexed collection of permutation
groups, $C^{[i]} \leq Sym(A^i)$, where the index corresponds to the gate arity.
A permutation clone $C$ can be seen as an inverse monoid $(C, \bullet, i_1)$, with
$\{i_n\mid n \in \N\}$ the set of idempotents.

In any permutation clone $C$, the unary part $C^{[1]}$ is found as a wreath product in all other parts, $C^{[1]} wr S_n \leq C^{[n]}$
because the wire permutations give us the right hand factor
while $f_1\oplus\dots\oplus f_n$ for $f_i \in C^{[1]}$
gives us the left hand factor.

\begin{example}
 Let $q$ be a prime power, $GF(q)$ the field of order $q$, 
$AGL_n(q)$ the collection of affine invertible maps of $GF(q)^n$ to itself.
We note that for all $m \in \N$, $AGL_n(q^m) \leq AGL_{nm}(q)$ \cite[p. 56]{mortimerdixon}.
For a prime $p$, let 
\begin{eqnarray*}\operatorname{Aff}(p^m) = \bigcup_{n\in \N} AGL_{nm}(p)
\end{eqnarray*}
 be the permutation clone of affine maps over $A=GF(p)^m$.
\end{example}

\begin{example}
 Let $A$ be a set. 
Let $\beta_1,\dots,\beta_n \in S_A$, 
so $(\beta_1,\dots,\beta_n) =\beta_1\oplus\dots\oplus\beta_n\in B_n(A)$.
Let 
\begin{eqnarray*}
\Deg(A)=\{\pi_\alpha \bullet (\beta_1,\dots,\beta_n) \mid n\in \N,\, \alpha\in S_n,\, \beta_i\in S_A\}
= \bigcup_n S_A wr S_n
\end{eqnarray*}
be the collection of \emph{essentially unary} or \emph{degenerate} maps.
\end{example}

\begin{example}
 Let $A$ be a set, let $o\in A$ be some arbitrary element. 
Let 
\begin{eqnarray*}
P_o(A)=\{f\in B(A) \mid f(o,\dots,o)=(o,\dots,o)\}
\end{eqnarray*}
 be the collection of \emph{$o$-preserving}  maps.
\end{example}

We say that a permutation clone $C \leq B(A)$ is \emph{borrow closed} if for all $f\in B(A)$, 
$f\oplus i_1 \in C$ implies that $f\in C$.
We say that a permutation clone $C \leq B(A)$ is \emph{ancilla closed} if for all $f\in B_n(A)$, 
$g \in C^{[n+1]}$ with some $a\in A$ such that for all $x_1,\dots,x_n \in A$, for all
$i\in\{1,\dots,n\}$, $f_i(x_1,\dots,x_n) = g_i(x_1,\dots,x_n,a)$ and
$g_{n+1}(x_1,\dots,x_n,a)=a$ implies that $f\in C$.
If a permutation clone is ancilla closed then it is borrow closed. 

For any prime power $q$, $\operatorname{Aff}(q)$ is borrow and ancilla closed.
For any set $A$, $\Deg(A)$ is borrow and ancilla closed.
For all $o\in A$, $P_o(A)$ is borrow but not ancilla closed.

\begin{lemma}
\label{lemma_borrowclosed}
 A permutation clone $C$ is borrow closed iff 
 if $f,f\oplus g \in C$ then $g\in C$.
\end{lemma}
\begin{proof}
 Let $g$ have arity $n$, $f$ have arity $m$.
 
 Suppose $C$ is borrow closed.
 Since $f\in C$ we know that $f^{-1} \in C$, so $f^{-1}\oplus i_n \in C$.
 Then $(f^{-1}\oplus i_n)\bullet(f \oplus g) = i_m\oplus g \in C$. By borrow closure, then $g\in C$ and we are done.
 
 For the other direction we always have that $f=i_1\in C$ and this is the definition of borrow closed.
 \qed
\end{proof}

Let $F\subseteq \OO(A)$ be a set of maps on $A$.
Define $PC(F) = \{f \in B(A) \mid \forall i \in 1,\dots,ar(f), \, f_i \in F\}$ as the set of permutations with component functions in $F$.
We can ask ourselves, when is $PC(F)$ a permutation clone?

From \cite{jerabek18} we know that permutation clones are defined as  dual to collections of weight maps, similar to the duality by relations for clones.

\begin{definition}
	Let $k\in \N$, let $(M,\cdot)$ be a commutative monoid. A \emph{weight map} is some map $w:A^k\rightarrow M$. Let $f \in B_n(A)$.
	Then $f$ \emph{respects}  $w$, $f \triangleright w$, 
	if for every $a \in A^{k\times n}$,
\begin{eqnarray}
	\prod_{i=1}^{n} w(a_{1i},\dots,a_{ki}) 
	    =  \prod_{i=1}^{n} w(f(a_{11},\dots,a_{1n})_{i},\dots,f(a_{k1},\dots,a_{kn})_{i}). \label{eqn_weight}
\end{eqnarray}	
Then $\Pol(w) = \{f \in B(A) \mid f  \triangleright w\}$ is then the set of \emph{polymorphisms} of the weight $w$.
\end{definition}

Such arrays will be written in square brackets, so $a$ above would be
\begin{equation}
a=\left[
\begin{array}{ccccc}
 a_{11} & \dots & a_{1n}   \\
 \vdots &          & \vdots  \\
 a_{k1} &  \dots & a_{kn} 
\end{array}
\right]
\end{equation}
We will often use shorthand, for instance if $a,b\in A^n$, then 
$\left[
\begin{array}{c}
 a   \\
b 
\end{array}
\right]
\in A^{2 \times n}$.

For simplicity we will write $w(a)$ for $\prod_{i=1}^{n} w(a_{1i},\dots,a_{ki})$ and 
\begin{eqnarray*}
f(a)= \left[
\begin{array}{c}
f(a_{11},\dots,a_{1n})   \\
\vdots \\
f(a_{k1},\dots,a_{kn}) 
\end{array}
\right]
\end{eqnarray*}
for the array obtained by applying $f$ to each row of $a$.
Then the equation (\ref{eqn_weight}) above becomes $w(a) = w(f(a))$.

\begin{definition}
 Let $c \in A^m$ and $f\in B_n(A)$. The \emph{controlled permutation} $CP(c,f)\in B_{m+n}(A)$
 applies $f$ to the last $n$ entries of its argument if the first $m$ entries match $c$.
\end{definition}

On the set $\{0,1\}$ the controlled permutation $CP(1,(01\;10))$ is a ternary map
that swaps the second and third arguments if the first argument is a 1. 
It is called the \emph{Fredkin gate} \cite{fredkintoffoli82} and has an important role in reversible computation.

\begin{example}
 Let $A=\{0,1\}$, $w:{0,1} \rightarrow (\N_0,+)$, $w(0)=0$, $w(1)=1$. Then $w$ counts the number of $1$s in a tuple,
 so $f \triangleright w$ iff $f$ conserves the number of $1$s, i.e. the weight of the tuple.
\end{example}

Such gates are  called \emph{conservative}, the Fredkin gate is an example.

Let $(\B,\wedge)$ be the two element monoid on $\{0,1\}$ with logical and  operation $\wedge$.

\begin{example}
 Let $(A,\leq)$ be a partially ordered set, $w:A^2 \rightarrow (\B,\wedge)$, 
 $w_\leq(a,b)=1$ if $a\leq b$ and $w_\leq(a,b)=0$ otherwise. Then for $f \in B_n(A)$,
 $f \triangleright w_\leq$ iff 
 $a\leq b \Rightarrow f(a)\leq f(b)$ iff
 $f$ 
 is a monotone map, i.e.\ an endomorphism and thus automorphism of $(A^n,\leq)$.
\end{example}

\begin{example}
 Let $(A,+)$ be an abelian group, $w:A^3 \rightarrow (\B,\wedge)$, 
 $w(a,b,c)=1$ if $a+b=c$ and $w(a,b,c)=0$ otherwise. Then for $f \in B_n(A)$, 
 $f \triangleright w$ iff $f$ 
 is a linear map, i.e.\ an automorphism of $(A,+)$.
\end{example}

\begin{example}
 Let $(A,+)$ be an abelian group, $w:A^4 \rightarrow (\B,\wedge)$, 
 $w(a,b,c,d)=1$ if $a+b=c+d$ and $w(a,b,c,d)=0$ otherwise. Then for $f \in B_n(A)$, 
 $f \triangleright w$ iff $f$ 
 is an affine permutation, i.e.\ the sum of an automorphism of $A$ and a constant map.
\end{example}

In particular, if $A=\Z_p$ for a prime $p$, then $\Pol(w) = \Aff(A)$.

\begin{definition}
\label{exampleRelationWeight}
Let $R \subset A^k$ be a relation.
Let $w_R:A^k\rightarrow (\B,\wedge)$ with $w_R(a)$ true iff $a\in R$.
Then $f\triangleright w_R$ iff every component $f_i$ respects the relation $R$, i.e.\ $f_i\triangleright R$ in the clone sense.
\end{definition}

We call such polymorphisms \emph{relationally defined}.
This type of weight will be the main concern of this paper.

Then $PC(\Pol(R \mid R \in \RR))$  $= \Pol(w_R\mid R \in \RR)$.

%

We note  that in order to check whether $f\in \Pol(w_R)$ we need only check for matrices 
with columns in $R$, that is, for matrices $M$ such that $w_R(M)=1$ and $w_R(f(M))=1$. By reversibility, the other matrices all map to $0$.

We note that we can look at the individual arities of a relationally defined permutation clone 
as the automorphism group of a relational structure.

\begin{lemma}
\label{lemma_AutR}
Let $\A=(A,R)$, $R \subseteq A^k$  be a relational structure.
Let $\A^n=(A^n,R^n)$ where $R^n= \{ (a_1,\dots,a_k) \mid a_i\in A^n,\, (a_{i1},\dots,a_{ik})\in R\; \forall i\in \{1,\dots,n\}\}$.
Then $\Pol(w_R)^{[n]} = Aut(\A^n)$.
\end{lemma}
\begin{proof}
 \begin{eqnarray*}
f \in \Pol(w_R)^{[n]} & \Leftrightarrow & f\in B_n(A),\, \forall a\in A^{k\times n},\, \\
&& \hspace{10mm}(a_{1i},\dots,a_{k,i})\in R\, \forall i,
      (f(a)_{1i},\dots,f(a)_{k,i})\in R\, \forall i \\
      &\Leftrightarrow & \forall a \in \A^n,\, f(a)\in \A^n \\
      &\Leftrightarrow & f \in \Aut(\A^n)
\end{eqnarray*}
\qed
\end{proof}

This also reminds us that  $f\triangleright w_R$ iff every array with columns in $R$ is mapped by $f$ to an
array with columns in $R$, and every array with at least one column not in $R$ is mapped to an array with at least one column not in $R$.


An algebraic structure $(S,+,*,0,1)$ is called a \emph{(commutative) semiring} if 
$(S,+,0)$ is a commutative monoid, $(S,*,1)$ is a commutative monoid and the two distribution laws
hold, i.e.\  for all $a,b,c \in S$
\begin{eqnarray}
a*(b+c) &= a*b + a* c \\
(a+b)*c &= a*c+b*c
\end{eqnarray}
Examples include commutative rings, fields, bounded lattices and $(S,+,+,0,0)$ when $(S,+,0)$ is
a join semilattice with $0$.

Je\v{r}\'{a}bek  defines a closure process for sets of weights
\cite[p. 11]{jerabek18}.

\begin{definition}
\label{defn_coclone}
 A set $D$ of weight maps is \emph{closed} if:
 \begin{enumerate}
  \item If $w:A^k\rightarrow M$ is in $D$, and $\rho:\{1,\dots,k\}\rightarrow \{1,\dots,l\}$, then the weight
  $w\circ \rho:A^l\rightarrow M$ is in $D$, where $w\circ \rho: (x_1,\dots,x_l) \mapsto w(x_{\rho(1)},\dots,x_{\rho(k)})$.
  \item If $w:A^k\rightarrow M$ is in $D$, and $\phi:M\rightarrow N$ is a commutative monoid homomorphism, then $\phi\circ w_k:A^k\rightarrow N$ is in $D$.
  \item If $w:A^k\rightarrow M$ is in $D$, and $N\leq M$ is a submonoid with $w(A^k) \subseteq N$, then $w:A^k\rightarrow N$ is in $D$.
  \item If $w_i:A^k\rightarrow M_i$ is in $D$ for all $i\in I$, then $w: A^k \rightarrow \prod_{i \in I} M_i$ is in $D$.
  \item The weight $c_1:A \rightarrow (\N_0,+)$ with $c_1(a)=1$ for all $a\in A$ is in $D$.
  \item The weight $\delta:A^2 \rightarrow (\B,\wedge)$ with $\delta(a,b)$ true iff $a=b$ is in $D$.
  \item If $w:A^k\rightarrow (M,\cdot)$ is in $D$, $k\geq 2$, and $(M,\boxplus,\cdot)$ is a semiring, then $w^\boxplus$ is in $D$,
  where $w^\boxplus:A^{k-1}\rightarrow M$, 
  \begin{align*}
  w^\boxplus(a_1,\dots,a_{k-1}) =  \underset{a\in A}\bigboxplus w(a_1,\dots,a_{k-1},a)      
  \end{align*}
 \end{enumerate}
\end{definition}
Je\v{r}\'{a}bek calls \cite[Defn 5.17]{jerabek18} a set of weight maps that are closed in this way a \emph{permutation co-clone}.

The important result is then the following, telling us that every permutation clone is defined
by a collection of weights.


\begin{theorem}[\protect{\cite[Thm 5.18]{jerabek18}}]
\label{theoremJerabek}
Let $C$ be a permutation clone. Then there exists a permutation co-clone $D$ such that for all $f \in B(A)$,  $f \in C$ iff $f  \triangleright w$ for all $w \in D$.
\end{theorem}

That is, a permutation clone is defined precisely by the permutation co-clone that it respects.
We write $\Inv(C)$, the \emph {invariants} of $C$, for the collection of weight maps that  are respected by a set of bijections $C \subseteq B(A)$
and $\Pol(D) \subseteq B(A)$, the \emph{polymorphisms} that respect $D$, for the set of bijections that respect a collection of weight maps $D$.
$D$ is closed iff $\Inv(\Pol(D))=D$ and $C$ is a permutation clone iff $\Pol(\Inv(C))=C$.
If $w$ is a weight mapping, we write $\Pol(w)$ to mean $\Pol(\langle w \rangle)$ where $\langle w \rangle$
is the permutation co-clone generated by $w$.
$\Pol$ and $\Inv$
By duality, a maximal permutation clone is defined by a minimal co-clone, which is generated by one weight.

In order to reduce confusion we will write $\Pol(R)$ to mean the clone of mappings that respect the relation $R$,
 $\Pol(w_R)$ to mean the permutation clone of bijections that respect the relation $R$ encoded as a weight function $w_R$, and $\PPol(R)=\Pol(w_R)$ (for permutation polymorphism) as shorthand, especially when talking about multiple relations. Thus $\PPol(R,S,T) = \Pol(w_R) \cap \Pol(w_S) \cap \Pol(w_T)$.


It is still unknown what precise form of weight is necessary in order to
define ancilla and borrow closed permutation clones.
Some sufficient conditions (for example, cancellative monoids) are known (see below), but they 
are not necessary \cite[Section 5.2]{jerabek18}.

\begin{theorem}
Let $w:A^k\rightarrow (M,\cdot)$ be a weight map, let $a\in A^k$ be such that $w(a)$ is cancellative in $M$.
Then $Pol(w)$ is borrow closed.
\end{theorem}
\begin{proof}
Suppose $f\oplus i_1 \in Pol(w)$, $f\in B_n(A)$.
Let $b \in A^{k\times n}$.
Let $d\in A^{k\times (n+1)}$ be the matrix consisting of  $b$ with the $n+1$-th column
being all $a$, that is, $d=b \oplus a$.
Then because $f\oplus i_1 \in Pol(w)$, $w(d) = w((f\oplus i_1)(d))$, thus
$w(b) \cdot w(a) = w((f\oplus i_1)(d))=w(f(b)\oplus a)= w(f(b))\cdot w(a)$.
Now because $w(a)$ is cancellative, this implies that $w(b)=w(f(b))$ so $f \triangleright w$
and thus $f\in Pol(w).$

Thus $Pol(w)$ is borrow closed and we are done.
\qed
\end{proof}
There is a similar result \cite{jerabek18} for ancilla closure.

\begin{theorem}
\label{thm_ancillaclosed}
Let $w:A^k\rightarrow (M,+)$ be a weight map,  such that $w(a,\dots,a)$ is cancellative in $M$ for all $a\in A$.
Then $Pol(w)$ is ancilla closed.
\end{theorem}

These conditions are sufficient but not neccesary. The following example shows this for 
borrow closure.

\begin{example}
 Let $A=\{0,1\}$, let $w:A\rightarrow \B^2$, $w(0)=(1,0)$ and $w(1)=(0,1)$.
 Neither of these two images is cancellative, but $\Pol(w)$ is borrow closed.
 This can be seen because the weight $v:A^2\rightarrow \B^2$ defined by
 
\begin{equation*}
\begin{array}{|c|c|}
\hline
00 &  10  \\
01 &  11   \\
10 &  00 \\
11 & 01 \\
\hline
\end{array}
\end{equation*}
 
 is in the permutation co-clone defined by $w$, as is the weight $\bar v:A^2\rightarrow \B$ with 
 $\bar v (a)= 1 $ iff $a=01$. It can also be seen that $w$ is in the permutation co-clone defined by $v$ and $\bar v$.
 Both $v$ and $\bar v$ have cancellative images, showing that $\Pol(v)$ is borrow closed.
 
 This is the set of permutations that fix the zero vector and the vector of all 1,
  i.e.\ $f\in \Pol(w)$ iff $f(0\dots 0) = 0\dots 0$ and $f(1\dots 1) = 1\dots 1$.
\end{example}

 \section{Some basic properties}

In this section we will look at some properties of the collection of 
relationally defined permutation clones. 
Firstly we will show that all permutation co-clones contain a relational part that is closed
as a co-clone, thus permutation co-clone closure is stronger than relational co-clone closure.
Then we will examine relational clone closure and show that 
maximal borrow closed permutation clones are either relationally defined or 
have cancellative weights.

\begin{lemma}
\label{lemma_relCloneClosure}
 Let $\RR$ be a collection of relations on $A$.
 Let $Wt(\RR) = \{w_R \mid R \in \RR\}$ be the collection of weights obtained from $\RR$.
 Let $D$ be a permutation coclone.
 Let $Rl(D) = \{R \mid w_R \in D\}$ be the collection of relations which have characteristic functions in $D$.
 Then $\langle \RR \rangle_{rc} \subset Rl(\langle Wt(\RR)\rangle_{pcc}$ where $rc$ means relation clone (coclone) closure and $pcc$ means permutation coclone closure.
 Moreover $Rl(\langle Wt(\RR)\rangle_{pcc})$ is a coclone.
\end{lemma}
\begin{proof}
We show that each of the coclone closure operations can be obtained from the 
permutation coclone closure operations.

 The permutation of entries follows from the first closure operation.
 
 Projection can be obtained by using the semigroup closure with the semigroup $(\B,\vee,\wedge)$.
 
 Let $R\subseteq A^k,\,S\subset A^l$ be two relations. Then $w_R$ and $w_S$ are two characteristic functions for these relations.
 Let $p_1:\{1,\dots,k\} \rightarrow \{1,\dots,k+l\}$ be defined by $p_1:i\mapsto i$ and
 $p_2:\{1,\dots,l\} \rightarrow \{1,\dots,k+l\}$ be defined by $p_2:i\mapsto k+i$.
 Then $w_R^{p_1}:A^{k+l} \rightarrow \B$ and $w_S^{p_2}:A^{k+l} \rightarrow \B$ are in the permutation
 coclone.
 We then create the product $(w_R^{p_1},w_S^{p_2}):A^{k+l} \rightarrow \B\times \B$ that is
 in the permutation coclone, then apply the homomorphism $\wedge: \B\times \B\rightarrow \B$
 to obtain a characteristic function of the cartesian product of $R$ and $S$.
 
 Let $R,S \subseteq A^k$. 
 Let $w_R$ and $w_S$ be the characteristic functions for these relations.
Then $(w_R,w_S): A^k \rightarrow \B \times \B$ is in the permutation coclone, we apply the homomorphism
$\wedge: \B\times \B\rightarrow \B$
 to obtain a characteristic function of the intersection $R \cap S$.
 
 The equality relation characteristic function $\delta$ is in the permutation coclone.
 Let $p:\{1,2\}\rightarrow \{1,2,3\}$ be the identity injection, so $\delta \circ p$ is the characteristic function
 of the equality relation is in a co-clone.
 \qed
\end{proof}

Thus the lattice of relationally defined permutation clones will be a
homomorphic image of the lattice of clones.

%
%
%
Relationally defined permutation clones have a stronger closure property.

\begin{lemma}
Every relationally defined permutation clone is borrow closed.
\end{lemma}
\begin{proof}
Let $w_R$ be the weight function from a relation $R$. 
If $R$ is non trivial, then $w_R(a)=1$ for at least one $a$.
$1$ is cancellative as it is an identity. 
Thus $Pol(w_R)$ is borrow closed.
\qed
\end{proof}

Relationally defined permutation clones have a stronger factoring property than 
borrow closure.

\begin{lemma}
 Let $C$ be a relationally defined permutation clone.
 If $f \oplus g \in C$, then $f,g\in C$.
\end{lemma}
\begin{proof}
Let $C =\PPol(\rho)$ be our relationally defined permutation clone.

Let $f$ have arity $n$ and $g$ have arity $m$.
 Let $a,b$ be two appropriately sized arrays of elements from $A$ such that every column
 is in the relation defining $C$.
 Apply $f \oplus g$ to $a\oplus b$ and the columns of $(f\oplus g)(a\oplus b)$ are all in the relation defining $C$. Thus since $a$ had columns in $\rho$ and the first $n$ columns of $(f\oplus g)(a\oplus b)$ are all in $\rho$ then $f(a)$ has columns in $\rho$ so $f\triangleright \rho$ and $f \in C$.
 Similarly $g \in C$.
 \qed
\end{proof}

We can determine the relations that give us the full permutation clone.

\begin{lemma}
\label{lemma_relationFull}
Let $A$ be a finite set.
 Let $\emptyset \neq R\subseteq A^k$ be a relation with $\PPol(R)=B(A)$.
 Then there exists an equivalence relation $E$ such that \[R = \{a\in A^k \mid a_i=a_j,\,\forall i,j: (i,j)\in E\}.\]
\end{lemma}
\begin{proof}
Let $k=1$, so $R\subseteq A$. 
Let $f$ be the cycle that rotates all the elements of $A$.
Then $f\triangleright w_R$ means that $f(R)=R$ so $R=A$.

We proceed by induction. Assume the result is  true for $k-1$.

Suppose there exist $a_1,\dots,a_{k-1}\in R$ such that for all $i$, $(a_i)_i \neq (a_i)_{i+1}$.
Then the $(k-1) \times k$ array 
$\left[ a_1\dots a_{k-1}\right]$ has $k$ distinct rows, so it can be mapped to
some $(k-1) \times k$ array $b$ such that the first column $b_1$ of $b$ is arbitrary. 
Thus $b_1\in R$ so $R=A^k$ and we are done.

If such a set of $a_i$ do not exist, assume without loss of generality that
for all $a\in R$, $a_{k-1}=a_k$. 
Let $w_R$ be the characteristic weight of $R$, let $(\B,\vee,\wedge)$ be a semiring
and note that $w^\vee_R(a_1,\dots,a_{k-1})=1$ iff $w_R(a_1,\dots,a_{k-1},a_{k-1})=1$.
From $w^\vee_R$ we can construct the weight
$v:A^k\rightarrow (\b,\wedge)^2$ such that
$v( a_1,\dots,a_{k})=(w^\vee_R(a_1,\dots,a_{k-1}, \delta(a_{k-1},a_{k})$.
Applying the AND homomorphism from $\B^2$ to $\B$ to $v$,
we find that $\wedge \circ v(  a_1,\dots,a_{k})=w_R( a_1,\dots,a_{k})$, so
$w^\vee_R$ and $w_R$ define the same permutation clone.

Then define $R_1\subset A^{k-1}$ such that $w_{R_1}=w^\vee_R$,. 
Apply the induction
hypothesis to find an equivalence relation $E_1$ on $\{1,\dots,k-1\}$. 
Extend $E_1$
by  $(k-1,k)$ to obtain an equivalence relation $E$ and we have proven our result.
\qed
\end{proof}

 A \emph{subelementary} monoid $S=C\cup N$ is a disjoint union of a cancellative monoid $C$ and a nil semigroup $N$. 
 Note that $N$ is an ideal of $S$.
 The two element boolean algebra $(\B,\wedge)$ is the simplest subelementary monoid; 
 a union of the one element cancellative monoid $\{1\}$ and
 the one element nil semigroup $\{0\}$. 
 For any subelementary monoid $S$, the map $h:S\rightarrow \B$ with $h(c)=1$ for all $c \in C$ and
 $h(n)=0$ for all $n \in N$ is a homomorphism. 
 Thus $\B$ is the only simple subelementary monoid.
 
It is known  (e.g.\ \cite{grillet2001}) that a finitely generated commutative monoid is a 
 subdirect product of a cancellative and a direct product of subelementary monoids, i.e. $C \times \prod_i S_i$.
 We obtain the following result immediately.
 
\begin{lemma}
  A meet irreducible permutation clone is defined by a cancellative or subelementary monoid weight.
\end{lemma}

For instance the conservative permutation clones are of this form, as
are the orthogonal permutation clones over the field  $(\Z_p,+,\cdot)$ by the weight
$w_O:\Z_p^2\rightarrow (\Z_p,+)$ with $w_O(x,y)=x\cdot y$.
 
 The following shows that for borrow closure (and thus for ancilla closure)
 we can annihilate the nil part of a subelementary monoid and still define the same
 permutation clone.
 
\begin{lemma}
\label{lemma_subelementaryNil}
 Let $w:A^k\rightarrow C \cup N$ be a subelementary weight. Let $\Pol(w)$ be borrow closed.
 Let $\phi: C \cup N \rightarrow C \cup \{0\}$ be the natural homomorphism annihilating the nil
 part. 
 Then $\Pol(w)=\Pol(\phi \circ w)$.
\end{lemma}

\begin{proof}
 Suppose this is not true, so there is some $f\in B_m(A)$ with $f \triangleright \phi \circ w$ and 
 $f \not\triangleright  w$.
 Let $a\in A^{k \times m}$ be the contradicting array, so $\phi \circ w(f(a))=\phi \circ w(a)$
 but $w(f(a))\neq w(a)$.
 Then $w(f(a)),\, w(a)\in N$.
 Let $b\in A^k$ such that $w(b)\in N$, let $h\in \N$ be the nil degree of $N$, so $\forall n \in N$, $n^h=0$.
 Let $\bar a \in A^{k \times (m+h)}$ be the array $a$ with $h$ copies of $b$ adjoined to the right.
 
 Then $w(\bar a) = w(a)\cdot w(b)^h=0$ and $w((f\oplus i_h)(\bar a)) = w(f(a) \cdot w(b)^h=0$
 so $(f\oplus i_h)\triangleright w$, so $(f\oplus i_h) \in \Pol(w)$.
 But $\Pol(w)$ is borrow closed, so this implies that $f\in \Pol(w)$ and we are done.
\end{proof}

The following shows that maximal borrow closed permutation clones
are defined either by cancellative weights or relational weights.

\begin{lemma}
\label{lemma_maximalSubEl}
Let $w:A^k\rightarrow (C \cup N,\cdot)$ be a nontrivially subelementary weight,
so $w(A^k)\cap N \neq \emptyset$ and $w(A^k)\cap C \neq \emptyset$. 
Let $\Pol(w)$ be borrow closed and maximal.
 Let $\phi: C \cup N \rightarrow  \{0,1\}$ be the natural homomorphism annihilating the nil
 part and mapping the cancellative part to $1$. 
 Then $\Pol(w)=\Pol(\phi \circ w)$.
\end{lemma}

\begin{proof}
We note that by Lemma \ref{lemma_subelementaryNil} we can take $N=\{0\}$ to be trivial.
Let $1\in C$ be the identity.

We proceed by induction.
Let $k=1$. Let $R\subseteq A$ such that
$w_R=\phi \circ w$.
 By the nontriviality of the subelementary weight, $R$ is nontrivial,
 and $\PPol(R)$ is a maximal borrow closed permutation clone,
 so $\Pol(w)=\Pol(\phi \circ w)$.

 Suppose this is not true for some $k>1$, but true for $k-1$.
 Then $\Pol(w)\neq\Pol(\phi \circ w)$ so $\Pol(\phi \circ w)=B(A)$ by maximality.
 
 From Lemma \ref{lemma_relationFull} above we know that
 the relation $R$ such that $w_R=\phi \circ w$ is defined
 by an equivalence relation $E$.
 
 If $E$ is the equality relation, then $R=A^k$ and $w(A^k)\subseteq C$, a contradiction.
 
 Without loss of generality, let $(k-1,k) \in E$.
 Thus $w(a_1,\dots,a_k)\in C$ implies that $a_{k-1}=a_k$.
 We define  a binary operation $+$  on $C\cup N$ such that $(C\cup N,+,\cdot)$ is a semiring.
 For all $c\in C$, $c+0=0+c=c$, for $c,d\in C$ let $c+d=0$.
 Then because $w^+$ has only one nonzero summand, 
 $w^+(a_1,\dots,a_{k-1})=w(a_1,\dots,a_{k-1},a_{k-1})$.
 
 Now let $w^+ \times \delta: A^k \rightarrow (C \cup N) \times \{0,1\}$ be defined by
 \begin{eqnarray*}
w^+ \times \delta (a_1,\dots,a_k) = (w^+(a_1,\dots,a_{k-1}), \delta(a_{k-1},a_k)).
\end{eqnarray*}
 Then $\psi:(C \cup N) \times \{0,1\} \rightarrow (C \cup N) $ defined by $\psi(c,1)=c$
 and $\psi(c,0)=0$ is a monoid homomorphism.
 Then $\psi \circ (w^+ \times \delta)=w$ so $w$ and $w^+$ define the same 
 permutation clone.

By the induction hypothesis, $\Pol(w^+)=\Pol(\phi \circ w^+)$.
We now define the relation  $R^+\subseteq A^{k-1}$
such that  $w_{R^+}= \phi \circ w^+$.
Then 
\begin{eqnarray*}
R= \{(a_1,\dots,a_k) \mid (a_1,\dots,a_{k-1}) \in R^+ \wedge a_{k-1}=a_k\},
\end{eqnarray*}
so
$R$ is in the relational clone defined by $R^+$. Similarly
\begin{eqnarray*}
R^+= \{(a_1,\dots,a_{k-1}) \mid (a_1,\dots,a_{k-1},a_{k-1}) \in R\}
\end{eqnarray*}
so $R$ and $R^+$ define the same relational clone.
Thus $\Pol(w)=\Pol(w^+)=\Pol(\phi \circ w^+)=\Pol(\phi \circ w)$
and we are done.  
 \qed
\end{proof}

Thus maximal borrow closed permutation clones are defined by a cancellative weight or a relational weight.
By Lemma \ref{lemma_relCloneClosure}
 we know that all maximal relationally defined permutation clones must be defined by relations corresponding to maximal clones, so by Rosenberg's result we know them all. We will investigate these in the next section.

The following example shows that  Lemma \ref{lemma_maximalSubEl}
 does not apply in general.

\begin{example}
\label{example_three}
 Let $M=\N_0\cup \{\infty\}$ be the subelementary monoid natural numbers with absorbing element $\infty$ adjoined.
 Let $A=\{0,1,2\}$ and $w:A\rightarrow M$ have $w(0)=0$, $w(1)=1$ and $w(2)=\infty$.
 Then $\Pol(w)$ will act as $\Cons(\{0,1\})$ on the set $\{0,1\}$.
 
 Let $\phi: \N_0\cup \{\infty\} \rightarrow  \{0,1\}$ be the natural homomorphism  mapping the cancellative part $\N_0$ to $1$ and $\infty$ to $0$. 
 Then  $\Pol(\phi \circ w)$ is $\PPol(\{0,1\})$, that is, the collection of permutations that respect the subset $\{0,1\}$.
 We find that $\Pol(w) \subsetneq \Pol(\phi \circ w) \subsetneq B(A)$.
\end{example}

\section{Learning from clones}

 As clone theory is well developed, there is a wealth of results that can be utilised for 
investigating the relationally defined permutation clones. We will see some of this here but there is  much more.

As a guiding principle we will use Rosenberg's classification of maximal clones, as all relational clones lie below one or more maximal clones, and thus all relationally defined permutation clones lie below one or more of the permutation clines defined by the Rosenberg relations.

First we will look at some tools, then we will look at the Rosenberg relations in more detail.

The property of being defined by a relation is not always clear. There are 
permutation clones that can be defined by relations as well as by cancellative weights,
such as the degenerate or essentially unary maps.

\begin{definition}
 Let $m,n\in \N$. 
 The Hamming graph $H(m,n)$ has node set $\{1,\dots,m\}^n$ with two nodes adjacent if they differ at exactly one
 position.
 Equivalently it is the direct product of $n$ copies of the complete graph $K_m$.
\end{definition}

\begin{theorem}[\cite{mifzal_zaiee_2021,praegerSchneider}]
\label{theoremHamming}
 The automorphism group $\Aut(H(n,m))$ of the Hamming graph is $S_m wr S_n$.
\end{theorem}

\begin{lemma}
Let $R=\{(a,b,c)\mid a=b \bigvee b=c\} \subseteq A^3$.
Let $w:A^2\rightarrow (\N_0,+)$ be defined by $w(a,a)=1$ and $w(a,b)=0$ for $a \neq b$.
Then $\Deg(A)=\Pol(w_R)=\Pol(w)$.
\end{lemma}

\begin{proof}
($\Deg(A)\subseteq \Pol(w_R)$):  Let $f\in \Deg(A)$, $f = \pi_\alpha \bullet (g_1,\dots,g_n)$ for some $g_i\in S_A$. 
Let $w_R(a)=1$ so each column of $a$ is in $R$. If $a_{1i}=a_{2i}$ then $f(a)_{1i^\alpha}=f(a)_{2i^\alpha}$ and similarly if $a_{2i}=a_{3i}$ then $f(a)_{2i^\alpha}=f(a)_{2i^\alpha}$ so every column of $f(a)$ is in $R$ and thus $w_R(f(a))=1$.

($\Pol(w_R) \subseteq \Deg(A) )$: Suppose $f\not\in Deg(A)^{[n]}$. Without loss of generality, let $f_1$ depend upon inputs 1 and 2, so there exists some $a_1,\dots, a_n,b_1,b_2\in A$ such that 
\begin{eqnarray}
f_1(a_1,\dots,a_n) &\neq& f_1(b_1,a_2,\dots,a_n) \\
f_1(a_1,\dots,a_n) &\neq& f_1(a_1,b_2,a_3,\dots,a_n) 
\end{eqnarray}
Then the matrix
\[
a=\left[
\begin{array}{ccccc}
 b_1 & a_2  & a_3 &\dots & a_n   \\
 a_1 & a_2  & a_3 &\dots & a_n  \\
 a_1 & b_2  & a_3 & \dots & a_n
\end{array}
\right]
\]
has every column in $R$, but the first column of $f(a)$ is not in $R$, so $f \not\in \Pol(w_R)$.


($\Pol(w) = \Deg(A) )$): Fix $n\in \N$. Define a graph on $A^n$ with edges $(a,b)$ iff the matrix with rows $a$ and $b$ has weight $n-1$. This occurs iff $a$ and $b$ differ in exactly one position. Thus this is the $n$-Hamming graph on the alphabet $A$. Any $f\in \Pol(w)$ must map matrices of weight $n-1$ to each other, so it preserves the edges of this graph.
From Theorem\ref{theoremHamming} we know that the automorphism group of the Hamming graph is the wreath product $S_A wr S_n$ which is precisely $Deg(A)$ so we are done.
\qed
\end{proof}

Thus while $Deg(A)$ can be defined by a cancellative weight, it can also be defined by a relational weight.

We cannot yet define the necessary conditions on a relation $R$ such that $\Pol(w_R)=\Pi$, but the understanding is that imbalances in the relation completion possibilities will force this,  such as that used in the previous lemma.
The first consideration is that all component maps $f_i$ for some $f\in C^{[n]}$ for some permutation clone are balanced \cite{toffoli80}. 
A map $g:A^n\rightarrow A$ is \emph{balanced} if the preimage $\vert g^{-1}(a)\vert = \vert A\vert^{n-1}$ for all $a\in A$.
It is easy to see that the component maps of wire permutations are balanced.

Let $C$ be a clone, let $\beta(C)$ be the collection of bijections that can be 
created from the maps in $C$.
Then $\beta$ induces an equivalence relation of clones.
Two clones will be $\beta$-equivalent if they
contain the same balanced maps.
We will see below that some balanced maps do not occur as a component map of a bijection.

There are some manipulations that determine permutation clones that include the permutation clone of interest, including the homomorphism closure and semiring closure from Definition \ref {defn_coclone} above.

The following is clear.
\begin{lemma}
Let $\rho_1,\rho_2 \subseteq A^k$ be two relations. 
Then $\PPol(\rho_1,\rho_2) \subseteq \PPol(\rho_1 \cap \rho_2)$. 
\end{lemma}

Let $R \subseteq A^k$ be a relation.
For any $i\in\{1,\dots,k\}$ let $c_{R,i}:A^{k-1}\rightarrow (\N,\cdot)$ be defined by 
\begin{equation}
c_{R,i}(a_1,\dots,a_{k-1}) = \vert \{x \mid (a_1,\dots,a_{i-1},x,a_i,\dots,a_{k-1})\in R\} \vert
\end{equation}

\begin{lemma}
\label{lemma_countingCons}
 For any relation $R \subseteq A^k$, $k\geq 2$, for any $i\in\{1,\dots,k\}$, $Pol(w_R) \subseteq Pol(c_{R,i})$.
\end{lemma}
\begin{proof}
 We show that for all $f\in B(A)$, $f\triangleright w_R \Rightarrow f \triangleright c_{R,i}$.
 
 Let $f\in B_n(A)$, $R \subseteq A^k$, $k\geq 2$. 
 Suppose  $f\triangleright w_R$.
 Without loss of generality, we will consider $c_{R,k}$.
 Let $a\in A^{(k-1)\times n}$. 
 For any $b\in A^n$ such that $\left[
\begin{array}{c}
 a  \\
 b
\end{array}
\right] \in R^n$, 
$\left[
\begin{array}{c}
 f(a)  \\
 f(b)
\end{array}
\right] \in R^n$.

Thus $\vert \{ b \mid \left[
\begin{array}{c}
 a  \\
 b
\end{array}
\right] \in R^n\}\vert = 
 \vert \{ b \mid \left[
\begin{array}{c}
 f(a)  \\
 b
\end{array}
\right] \in R^n\}\vert$. 
But this is simple rewriting that $c_{R,k}(a)=c_{R,k}(f(a))$ which means that $f\triangleright c_{R,k}$, which is what we wanted.
 \qed
\end{proof}

There is an alternative proof using the closure results from Theorem \ref{theoremJerabek}.
\begin{proof}
Note that $\{0,1\} \leq (\N,\cdot,1)$ as a monoid. 
Treat $w_R:A^k\rightarrow (\N_0,\cdot,1)$
as the weight map.
Then one semiring addition on $\N_0$ is the normal addition and 
\begin{equation*}
c_{R,k}(x_1,\dots,x_{k-1})= w_R^+= \sum_{a\in A} w(x_1,\dots,x_{k-1},a)
\end{equation*}
Thus $c_{R,k} $ is in the permutation coclone defined by $w_R$ so $\Pol(w_R) \leq \Pol(c_{R,k})$.
This holds for all indices using the wire permutations, so we are done.
\qed
\end{proof}

Note that $c$ maps to the natural numbers under multiplication.
This has a few implications. If $c_{R,k}(a)=0$ for some $a$, 
this absorbing zero gives an example of a properly subelementary weight.
If $c_{R,k}(a)\neq0$ for all $a$, then we obtain something more like the conservative permutation clones
defined above, but (for $k\geq 3$) on tuples rather than individual elements of $A$.
We will have a number of conservative weights based upon the prime factors of $c_{R,k}(a)$, i.e.\ the prime numbers less that $\vert A \vert$.

\begin{example}
Let $A=\{0,1,2\}$. Let $\rho=\{01,10,11\} \subset A^2$.
Then $c_{\rho,1}=c_{\rho,2}$ with $c_{\rho,1}(0)=1$, $c_{\rho,1}(2)=2$ and $c_{\rho,1}(3)=0$.
By the mapping $\infty\mapsto 0$, $n\mapsto 2^n$ we see that this is the same weight that we saw in
Example \ref{example_three}.
\end{example}

An opposite construction, from conservative weights to relational weights, can also be undertaken.
\begin{lemma}
\label{lemma_maxCons}
 Let $w:A^k\rightarrow \N_0$ be a weight. Let $m\in \N$ be maximal in $w(A^k)$.
 Let $\rho=\{a\in A^k\mid w(a)=m\}$.
 Then $\Pol(w) \subseteq \PPol(\rho)$.
\end{lemma}
\begin{proof}
 Suppose $f\in \Pol(w)^{[n]}$.
 Let $a \in A^{k \times n}$ with all columns taken from $\rho$, i.e.\ $a \in \rho^n$.
 Then $w(a) = nm$ and since $f\triangleright w$, $w(f(a))=nm$. Thus each column of $f(a)$
 has weight $m$ so each column is in $\rho$, so $f\in \PPol(\rho)$.
 \qed
\end{proof}

It is clear that the same result holds for the relation defined by the
tuples that map to the minimal value in $\N_0$.

There is an alternate proof that uses the coclone closure operations.
\begin{proof}
 Let $\psi:\N_0 \rightarrow \N_0$ be the homomorphism $\psi(x)=mx$.
 
 Then the weight $w \times \psi\circ c_1:A^k \rightarrow \N_0 \times \N_0$ is in the permutation coclone.
 Let $\phi: \N_0 \times \N_0 \rightarrow \Q\cup \{\infty\}$ be the homomorphism
 $\phi(x,y)=x/y$.
 Then the weight $\phi\circ (w \times \psi\circ c_1)$ is in the permutation coclone.
 
 Let $M\subseteq \Q\cup \{\infty\}$ be the image of $\phi\circ (w \times \psi\circ c_1)$.
 It will consist of $1$ and rational numbers between $0$ and $1$.

Let $\zeta: M \rightarrow \B$ be defined by $\zeta(1)=1$ and $\zeta(x) =0$ otherwise.
Then $\zeta$ is a homomorphism and $\zeta \circ \phi\circ (w \times \psi\circ c_1)$
is the characteristic function of the relation $\rho$.
\end{proof}

This proof makes clear that this result applies for any value $m$ that is not the product of other values in $w(A^k)$, i.e.\ for $a_i \in w(A^k)$, $\prod_{i=1}^n a_i = m^n$ iff $a_i=m$ for all $i$.

\subsection {Partial Orders}

A partial order $\A=(A,\leq)$ is  \emph{exponentially decomposable} 
if $\A = \C^n$ for some partial order $\C=(C,\leq)$.

\begin{lemma}[ \cite{Jonsson_1982}]
\label{lemma_partial_order}
 Let $(A,\leq)$ be an  exponentially indecomposable partial order.
 Then $\PPol(\leq)^{[n]}= \Aut ( A^n,\leq) = S_n wr Aut(A)$.
\end{lemma}
\begin{proof}
 The first equality follows from  Lemma \ref{lemma_AutR}, the second from Jonsson.
 \qed
\end{proof}

\begin{example}
\label{ex_totalorder}
 Let $(A,\leq)$ be a total order. 
Then $\Aut(A)$ is trivial, so $\Aut ( A^n,\leq)$ is the symmetric group on $n$ elements, 
$S_n$ acting on coordinates,
which gives us $\Pi_n$, so $\PPol(\leq)=\Pi$ the trivial permutation clone. 
\end{example}

This is in stark contrast to the situation for clones, where $\Pol(\leq)$ for a bounded partial order is a maximal
clone \cite{Rosenberg1970,pinskerthesis2002}.

\subsection{Unary Relations}

The simplest relations are unary relations  $\rho \subseteq A$.

\begin{lemma}[\cite{boykett2021}]
 Let $\rho \subseteq A$.
Then $\PPol(\rho)$ is a maximal borrow closed permutation clone.
\end{lemma}

The following result is clear.
\begin{lemma}
 Let $\rho_1,\rho_2 \subseteq A$.
Then $\PPol(\rho_1,\rho_2) \leq \PPol(\rho_1 \cap\rho_2) $.
\end{lemma}

Thus the collection of sub permutation clones defined by unary relations is isomorphic
to the collection of subsets of $\Pow(A)$ that are closed under intersection.
The smallest such permutation clone is $\PPol(\Pow(A))$.
Each unary relation defined permutation clone is defined by a collection of subsets of $A$ closed under intersection. These are called Moore Collections and can be counted using sequence 
A102894 of the Online Encyclopedia of Integer Sequences.

Chapter 16 of \cite{lau_2006} investigates the case of $\Pol(\{a\} \mid a \in Q)$ for $Q\subseteq A$.
In particular we know that the maximal subclones of these clones are intersections with the 
known maximal clones.

It is clear that $\Cons(A) \leq \PPol(\Pow(A))$.

\begin{lemma}
Let $A$ be of prime order $p \neq 2$.
Then  $\PPol(\Pow(A)) \cap \Aff(A) = \Pi$
\end{lemma}

\begin{proof}
 Let $f \in \PPol({\Pow}(A)) \cap \Aff(A)^{[n]} $.
 Then $f(0)=0$ so $f$ is linear.
 For all $i$, $f(e_i^{0,1}) \in \{0,1\}^n$ so the matrix of $f$ is a 0-1 matrix.
 $f(11\dots1)=11\dots 1$ so each row of the matrix contains $1 \mod p$ 1s.
 
 Let $x_i=1$ for all $i$ except $i=j,k$, where $x_j=x_k=0$.
 Then $f(x) \in \{0,1\}^n$. If the matrix of $f$ had a 1 in column $j$ and column $k$ for the same
 row $l$, then $f(x)_l=-1$, but $-1 \not\in \{0,1\}$ so this is a contradiction. 
 Thus the matrix has only one 1 in each row and is a 
 permutation matrix, so $f$ is a wire permutation, 
 and $\PPol(\Pow(A)) \cap \Aff(A) = \Pi$ \qed
\end{proof}

\subsection{Equivalence Relations}

In this section we discuss the case when a permutation clone preserves an equivalence relation.

An equivalence relation is a binary relation that is reflexive, transitive and symmetric, that is,
for all $x,y,z\in A$, $x \rho x$, $x \rho y \wedge y \rho z \Rightarrow x \rho z$ and 
$x \rho y \Leftrightarrow y \rho x$.
An equivalence relation is equivalent to a partition of $A$, where each part of the partition is
an equivalence class.

The following result shows a  class of maximal borrow and ancilla closed permutation clones.

\begin{theorem}
 Let $\rho$ be an equivalence relation with all equivalence classes
 of the same size.
 Then $\PPol(\rho)$ is a maximal borrow  and ancilla closed permutation clone.
\end{theorem}
\begin{proof}
 For any $n\in \N$, $\rho^n$ is an equivalence relation with all equivalence classes the same size.
 Thus $\PPol(\rho)^{[n]} = \Aut(\rho^n)$.
 By \cite{LPS87} we know that $\Aut(\rho^n)$ is maximal in $S_{A^n}=B_n(A)$.
 
 Let $f\not \in \PPol(\rho)$ of arity $n$.
 Then $\langle \PPol(\rho),f\rangle^{[n]}= B_n(A)$.
 By borrow (or ancilla) closure, this means that for all $m\leq n$, $\langle \PPol(\rho),f\rangle^{[m]}= B_m(A)$.
 
 Then $\langle \PPol(\rho),f\rangle^{[1]}= B_1(A)$. Thus there is some unary map $g$ that
 breaks the equivalence relation $\rho$. Then for any $n$, $g\oplus i_n$ is an $(n+1)$-ary map
 that is not in $\PPol(\rho)^{[n+1]}$, so $\langle \PPol(\rho),g\rangle^{[n+1]}= B_{n+1}(A)$.
 Thus $\langle \PPol(\rho),g\rangle= B(A)$, so $ \PPol(\rho)$ is maximal.
 \qed
\end{proof}

\begin{lemma}
 Let $\rho$ be an equivalence relation on $A$. Let $f\triangleright \rho$.
 Then $f$ is well-defined on $A/\rho$.
\end{lemma}

Let $\epsilon:B(A) \rightarrow B(A / \rho)$ be the natural homomorphism.
Then $\ker \epsilon$ is the set of bijections that are mapped to the identity.
This is not a permutation clone, as it does not include the wire permutations $\Pi$,
but it is a reversible iterative algebra.

\begin{lemma}
 Let $\rho$ be an equivalence relation on $A$. Let $C \leq \PPol(\rho)$
 be a sub-permutation clone.
 Then $\ker \epsilon ^{[n]}$ is a normal subgroup of $C^{[n]}$ and 
 $C^{[n]}$ is a group extension of $\epsilon C$ by $\ker \epsilon ^{[n]}$.
\end{lemma}

\begin{example}
 Let $A=\{0,1,2,3\}$ and $\rho=01\mid 23$.
 Then $\ker \epsilon^{[n]} = B_n(\{0,1\})\bullet B_n(\{2,3\})$ and $\epsilon (\PPol(\rho))\cong B(\{0,1\})$
 with the isomorphism  $[0]\mapsto 0$ and $[2]\mapsto 1$.
 
 We see that $CP(\{0,1\},(01)(23)) \in \ker \epsilon$ because the permutation is
 within the equivalence classes.
 Let $f=CP(\{0,1\},(02)(13))$. 
 Then $f \triangleright \rho$, $\epsilon(f)=CP(0,(01\,10))$ is the Fredkin gate.
\end{example}

This is a very clean example, as all the equivalence classes are the same size, so
$\PPol(\rho)$ can permute them.

\begin{example}
 Let $A=\{0,1,2,3\}$ and $\rho=0\mid 123$.
 Then $\ker \epsilon^{[n]} = B_n(\{0\})\bullet B_n(\{1,2,3\})$ and 
 $\epsilon (\PPol(\rho))\leq Cons_{[0]} \leq B(\{[0],[1]\})$ because $\PPol(\rho)$ cannot map
 between the equivalence classes because of their different sizes (Lemma \ref{lemma_countingCons}).
\end{example}

Thus the lattice of (relationally defined) permutation clones below the
permutation clone defined by an equivalence relation can be determined using
inverse monoid extensions, see e.g.\ \cite{CARVALHO_GRAY_RUSKUC_2011} for some details. 

\subsection{Affine Maps}

The following result has been found various times, see for instance 
\cite{Demetrovics_Bagyinszki_1982} for the inclusion diagram for $\vert A \vert=3$, 
Proposition 2.9 in \cite{szendrei1986clones} or \cite{salomaa_1964}.

\begin{lemma}
 Let $\vert A \vert$ be a prime. Then 
 $\Aff(A)$ has finitely many relationally defined sub-permutation clones.
 \end{lemma}
 

Note that this does not show that there are only finitely many sub-permutation clones of the affine
permutation clone. There are sub-permutation clones of the affine maps that are not defined by relations.
Let $A=\Z_p$ be of prime order, let the weight $w:A^3 \rightarrow A$ be defined by
$(a,b,c) \mapsto (a-c)\cdot (b-c)\in \Z_p$. Then $\Pol(w)\cap \Aff(A)$ will be the collection of affine maps 
for which the linear part of the map is orthogonal.
This cannot be a relationally defined permutation clone.

\subsection{Central Relations}

Let $\rho \subseteq A^h$ be a relation. The \emph{center} of $\rho$ is
$c(\rho)= \{a \in A \mid \forall a_2,\dots,a_h \in A,\, (a,a_2,\dots,a_h)\in \rho\}$.

 A relation $\rho \subseteq A^h$ is \emph{central} iff it is totally reflexive, totally symmetric and has a non void center that is a proper subset of $A$.

\subsection{Degenerate or Essentially Unary Maps}

We saw above that the essentially unary maps have the interesting property that they can be 
defined as a relational as well as by a cancellative weight.

The maps are known by two names. 
In the field of reversible computation, they are known as degenerate maps \cite{aaronsonetal15}, 
probably because they do not involve any interaction between their inputs.
In the clone theory field, they are known as essentially unary maps, because the output is only dependent upon a single input, thus they are generalise being unary.
In clone theory, there is some use of the transformation monoid that is at the core of a clone
of essentially unary maps. In our case, that will be a collections of permutations of our set $A$,
a permutation group.

\begin{theorem}
\label{theorem_neq_degenerate}
 Let $A$ be a finite set with three or more elements.
 Let $\eta = \{(a,b) \mid a \neq b\} \subset A^2$ be the binary non-equality relation.
 Then $\PPol(\eta)=\Deg(A)$.
\end{theorem}

\begin{proof}
 Let $f\in \Deg(A)^{]n]}$. Let 
 $ \left[
\begin{array}{c}
x \\
y
\end{array}
\right]  \in \eta^n$.
Then for all $i$, $f(x)_i$ depends only upon one input, wlog, $x_j$. Then $f(x)_i = f(y)_i$ 
iff $x_j=y_j$, but $x_j\neq y_j$ so $f(x)_i \eta f(y)_i$. 
Thus 
 $ \left[
\begin{array}{c}
f(x) \\
f(y)
\end{array}
\right]  \in \eta^n$, so $f \triangleright \eta$. Thus $\Deg(A) \leq \PPol(\eta)$.

Let $A=\{0,\dots,m-1\}$.
Choose some $n$.
Let $\Gamma$ be the graph $(A^n,\eta)$, let $G=\Aut(\Gamma)$.
$G$ is transitive on $A^n$ so $G_{e_0}$ is index $\vert A \vert^n$ in $G$.
$G_{e_0}$ is transitive on $(A \setminus \{0\})^n$ so
$G_{e_0,e_1}$ has index $(A \setminus \{0\})^n=(\vert A\vert-1)^n$ in $G_{e_0}$.

Thus $G_{e_0,e_1,\dots,e_{m-1}}$ has index $\vert A \vert^n (\vert A\vert-1)^n \dots 2.1 = (\vert A\vert !)^n$
in $G$.
Let $H=G_{e_0,e_1,\dots,e_{m-1}}$.
$H$ fixes every subset of $A$.
Then for $x\in \{0,a\}^n$, $a\not\in\{0,1\}$, $(e_i^{0,1},x)\in \Gamma$ or equivalently, $ \left[
\begin{array}{c}
e_i^{0,1} \\
x
\end{array}
\right]  \in \eta^n$ iff $x=e_i^{a,0}$ or $x=e_a$.
Then if $f(e_i^{0,1})$ has more than one $1$, there will be more than two neighbours of $f(e_i^{0,1})$
in $\{0,a\}^n$, which would mean that $f$ is not an automorphism of $\Gamma$. Thus
$f(e_i^{0,1})=e_j{0,1}$ for some $j$.

Thus $H$ fixes the set $\{e_i^{0,1}\mid i=1,\dots,n\}$ as a set, and is transitive on them.
Let $K\leq H$ fix these pointwise, so $K$ is of index $n!$ in $H$.

Then by the adjacency argument above, for all $a\not\in\{0,1\}$, $e_i^{a,0}$ is fixed by all $f\in K$ and similarly for all $a\neq b$, $e_i^{a,b}$ is fixed by $K$.

Lt $f\in K$. Let $x\in \{0,1\}^n$, $a\not\in\{0,1\}$. Suppose $x_i=1$. Then $(x,e_i^{a,0})\in \Gamma$ so 
$(f(x),e_i^{a,0})\in \Gamma$ thus $f(x)_i=1$. Similarly for $x_i=0$ so we see that $K$ fixes $\{0,1\}^n$ pointwise. Similarly $K$ fixes all $\{a,b\}^n$ pointwise.

Let $x\in A^n$.
Assume $x_i= a\in A$. Let $c\in A$, $c\neq a$.
Then for all $j$, let $\bar x_j=c$ if $x_j=a$, $\bar x_j=a$ otherwise.
Then $\bar x$ is fixed by $K$.
Then $(x,\bar x) \in \Gamma$ so $(f(x),f(\bar x)=(f(x),\bar x) \in \Gamma$, so $f(x)_i \neq c$.
Since $c$ was arbitrary and not equal to $a$, we see that $f(x)_i=a$.
Thus $f$ fixes all of $A^n$ so $K$ is trivial.

Thus the size of $G$ is $(G:H)(H:K) = ((\vert A\vert !)^n)(n!)$ which is equal to the size of $\Deg(A)^{[n]}$ so by inclusion, the two are the same.

Thus $\PPol(\eta)=\Deg(A)$ and we are done.
\qed
\end{proof}

When $A$ has two elements, the $\eta$ relation corresponds to the self-duality relation and so 
$\PPol(\eta)$ is not degenerate.

The following notes are taken from \cite[Theorem 2.2]{szendrei2012}

Let $S$ be the clone of mappings that are either essentially unary or nonsurjective, called S{\l}upecki's clone.
Then  for $2\leq i \leq \vert A\vert$, there exist a series of subsets, $S_i \subseteq \OO(A)$, that include 
all maps with range of order $i$ or less, together with the essentially unary maps.
$S_{\vert A\vert} = \OO(A)$ and $S_{\vert A\vert-1} = S$.
Then there exists a uniquely defined $S_1$ and we define $S_0$ to be the transformation monoid on $A$,
so that $S_0\subsetneq S_1\subsetneq \dots \subsetneq S_{\vert A\vert-1}  \subsetneq S_{\vert A\vert} =\OO(A)$ is an unrefinable strict chain, meaning that $S_i\leq S_{i-1}$ is a maximal subclone.

\begin{definition}
Let $A$ be a finite set.

Define $\beta = \{(a_1,a_2,a_3,a_4) \in A^4\mid a_1=a_i \wedge a_j=a_k, \, \{1,a_i,a_j,a_k\}=\{1,2,3,4\}\}$.

For $2\leq m \leq \vert A\vert$, define
$\iota_m=\{(a_1,\dots,a_m) \mid \exists i\neq j:\:a_i=a_j \}$. That is, all $m$-tuples that contain at least one repetition.
\end{definition}

\begin{theorem}[\protect{\cite[Theorem 2.2]{szendrei2012}}]
\label{thm_szendrei_slupecki}
 $S_1=\Pol(\beta)$ and for all $m\leq \vert A\vert$; $S_{r-1}=\Pol(\iota_r)$.
\end{theorem}

When we look at these relations for defining permutation clones, we note that the various $S_i$ contain
unbalanced maps and essentially unary maps. So the collection of permutation clones defined by these relations collapses.

\begin{definition}
 Let $R \subseteq A^k$ be a relation.
 We say that $R$ is \emph{totally symmetric} if for all permutations $\alpha \in S_k$,
 $(a_1,\dots,a_k) \in R$ iff $(a_{\alpha(1)},\dots,a_{\alpha(k)}) \in R$.
\end{definition}

\begin{theorem}
\label{thm_include_iota}
Let $3 \leq m \leq \vert A\vert$. 
Let $\rho\subseteq A^m$ be a totally symmetric relation that includes $\iota_m$.
 Then $\PPol(\rho)\leq\PPol(\iota_m)=Deg(A)$.
\end{theorem}
\begin{proof}
First we show that $\PPol(\iota_m)=Deg(A)$ using  Theorem \ref{thm_szendrei_slupecki} above.
  $S_{m-1}=\Pol(\iota_m)$ consists of essentially unary maps and those with image
 of size less than or equal to $m-1$. Maps with image of size less than $\vert A\vert$ cannot be balanced, 
 cannot appear as components of bijections. Thus all components maps 
 of $\PPol(\iota_m)$ are essentially unary so all elements of  $\PPol(\iota_m)$ are degenerate.
 
 Applying Lemma \ref{lemma_countingCons} to $\iota_m$ we obtain
 \begin{eqnarray}
c_{\iota_m,m}(a_1,\dots,a_{m-1})&=&
\left\{
\begin{array}{l}
(m-1) \mbox{ if } a_1,\dots,a_{m-1} \mbox{ are all distinct}\\
\vert A \vert \mbox{ otherwise }
\end{array}
\right.
\end{eqnarray}

All tuples $(a_1,\dots,a_{m}) \in \rho \setminus \iota_m$ contain no repetitions, they are sets of $m$ 
distinct elements of $A$. Thus

 \begin{eqnarray}
c_{\rho,m}(a_1,\dots,a_{m-1})&=&
\left\{
\begin{array}{l}
m \mbox{ or more, if } a_1,\dots,a_{m-1} \mbox{ are all distinct } \\
\hspace{15mm}\wedge(a_1,\dots,a_{m-1},a)\in \rho \setminus \iota_m \exists a\in A\\
(m-1) \mbox{ if } a_1,\dots,a_{m-1} \mbox{ are all distinct}\\
\vert A \vert \mbox{ otherwise }
\end{array}
\right.
\end{eqnarray}

Applying Lemma \ref{lemma_maxCons} to this weight, using the fact that $\vert A\vert \geq m > m-1$, 
we let $\gamma = \{(a_1,\dots,a_{m-1})\mid c_{\rho,m})(a_1,\dots,a_{m-1})=m-1\}$.
Then $\PPol(\rho) \leq \PPol(\gamma)$
Thus $\gamma$ consists of tuples that contain no repeats.
Let $\theta\subset A^2$ consisting of the first two entries of $\gamma$, so $\PPol(\gamma) \leq \PPol(theta)$. But $\theta=\eta$, so
by Theorem \ref{theorem_neq_degenerate} we know that $\PPol(\rho) \leq \PPol(\eta)=\Deg(A)$.
\qed
\end{proof}

If $\rho$ properly contains $\iota_m$ then $\rho \setminus \iota_m$ will define a set of $m$-sets of $A$.
Then $\PPol(\rho) = \langle \Gamma \rangle$ where $\Gamma \leq S_A$ is the permutation group on $A$
that fixes this set of $m$-sets.
This can also be seen as the automorphism of the $m$-uniform hypergraph with hyperedges given by the $m$-sets.

In any case, there are only a finite number of 
sub-permutation clones of $\PPol(\rho)$.

\subsection{$h$-regular // $h$-universal // $h$-generated }

The $h$ generated relations are defined by a complex reduction.
However the definition implies that these relations are totally reflexive and totally symmetry, thus they include $\iota_m$ and by Theorem \ref{thm_include_iota} above, the permutation clone generated by a $h$-regular relation is essemtially unary.

\subsection{Self dual}

Self dual mappings on a set $A$ of order $n$ respect a permutation of $A$
of prime order $p$ with no fixed points. Thus $p$ divides $n$.

If $A$ is of prime order, then these is one cycle and every map that respects that
cycle is balanced. Thus every map in the clone could be a component map of a bijection.

We have no further results in this case.

\section{ $\Pol(w_R)$ on a two element set}

The collection of clones on a two element set has been completely described by Post. 
Thus we can determine all relationally defined permutation clones on a two element set by inspecting each 
of the clones in the Post lattice.

\begin{table}[htp]
\caption{The $\beta$ classes of clones on the two element set, with their defining relations.}
\label{table_binary}
\begin{center}
\begin{tabular}{|c|c|c|}
\hline
Name & Description & Defining Relation \\
\hline
$\top$ & All maps & \{0,1\}\\
\hline
$D$ & Self dual & \{01,10\}\\
\hline
$DP$ & \makecell{Self dual, 0-preserving \\and 1-preserving} & \{0110,1010\}\\
\hline
$A$ & Affine & \{0000,0011,0101,0110,1001,1010,1100,1111\}\\
\hline
$AP_0$ & \makecell{Affine and 0-preserving,\\  i.e. linear} & \{000,011,101,110\}\\
\hline
$AD$ & Affine Self dual & $\{0000,0011,0101,0110,1001,1010,1100,1111\} \times \{01,10\}$\\
\hline
$AP_1$ & Affine and 1-preserving & \{00001,00111,01011,01101,10011,10101,11001,11111\}\\
\hline
$AP$ & Linear and 1-preserving  & \{0001,0111,1011,1101\}\\
\hline
$P_0$ & 0-preserving  & \{0\}\\
\hline
$P_1$ & 1-preserving  & \{1\}\\
\hline
$P$ & 0- and 1-preserving  & \{01\}\\
\hline
$U$ & \makecell{Essentially Unary \\ (Degenerate)}  & \{000,001,011,100,110,111\}\\
\hline
$\Pi$ & Projections & something trivial\\
\hline
\end{tabular}
\end{center}
\label{default}
\end{table}%


%

\begin{theorem}
\label{thm_binary}
 There are exactly 13 relationally defined permutation clones on a set of order 2.
\end{theorem}

\begin{proof}

We start by showing that the clones defined by $T_0^2,\,M, T_1^2$ all lie in the same
$\beta$-class as the projections.

Let us consider $M$ first. 
A bounded partial order on two elements is a total order.
From Lemma \ref{lemma_partial_order} we know that $\PPol(\leq)=\Pi$.

$T_0^2=\{01,10,00\}$. Let $f \in \Pol(T_0^2)^{[n]}$, suppose $f$ is balanced. 
For all $x\in A^n$, 
$\left[\begin{array}{c}
x\\
\bar x
\end{array}\right] \in (T_0^2)^n$, so 
$\left[\begin{array}{c}
f(x)\\
f(\bar x)
\end{array}\right] \in T_0^2$ so one of them is zero. In order for $f$ to be balanced, 
one of them must be 0 and the other 1, so $f(\bar x) = \bar f(x)$, so $f\in D$.
However $ \Pol(T_0^2)\cap D \subseteq M$ and the only balanced maps in $M$ are projections.

Similarly $\Pol(T_1^2)$ contains no nontrivial balanced maps, so each of these 
clones lie in the same $\beta$ class as the projections.

 Now we show that $(U,UD) \in \beta$. The only balanced essentially unary maps on $A$ are 
 the identity and NOT. Which are self-dual, so $\beta(U)=\beta(UD)$.
 
We now look at the remaining 13 classes and show that each inclusion is proper, by demonstrating a bijection that is in one but not the other. 
We will do this from the top down in the inclusion lattices in Figure \ref{figure_lattice_binary}.
We represent our permutations as two columns with the domain on the left and the image of each element in the domain on the right, except when we can reprsent them by affine maps.

\begin{itemize}
\item $B(A)$: The map 
\begin{equation*}
\begin{array}{|c|c|}
\hline
000  &   001  \\
001 &   010   \\
010  &    000 \\
011 & 100 \\
100 & 011 \\
101 & 101 \\
110 & 111\\
111 &   110\\
\hline
\end{array}
\end{equation*}
is neither self dual, affine, 0-fixing nor 1-fixing.
\item $D$: The map 
\begin{equation*}
\begin{array}{|c|c|}
\hline
000 &  111  \\
001 &  001   \\
010 &  010 \\
011 & 011 \\
100 & 100 \\
101 & 101 \\
110 & 110\\
111 &   000\\
\hline
\end{array}
\end{equation*}
is  self dual, but not affine, 0-fixing nor 1-fixing.

\item $DP$: The map 
\begin{equation*}
\begin{array}{|c|c|}
\hline
000 &  000  \\
001 &  001   \\
010 &  101 \\
011 & 011 \\
100 & 100 \\
101 & 010 \\
110 & 110\\
111 &   111\\
\hline
\end{array}
\end{equation*}
is  self dual, 0-fixing and 1-fixing, but not affine.
\item $A=\Aff(2)$: The map 
$x \mapsto 
\left[
\begin{array}{ccc}
1  & 1  & 0  \\
0  & 1  &0   \\
  0&  0 &  1 
\end{array}
\right] x + 
\left[
\begin{array}{c}
1   \\
1    \\
1 
\end{array}
\right]
$
is affine, but not linear, 1-fixing, nor self dual.

\item $AP_0$: Affine and fixing zero means linear. The map 
$x \mapsto 
\left[
\begin{array}{ccc}
1  & 1  & 0  \\
0  & 1  &0   \\
  0&  0 &  1 
\end{array}
\right] x $
is linear, but not 1-fixing, nor self dual.

\item $AP_1$: Affine and fixing zero means linear. The map 
$x \mapsto 
\left[
\begin{array}{ccc}
1  & 1  & 0  \\
0  & 1  &0   \\
  0&  0 &  1 
\end{array}
\right] x + 
\left[
\begin{array}{c}
1   \\
0    \\
0 
\end{array}
\right]$
is 1-fixing, but not linear, so not in $AP$.

\item $AD$: The map 
$x \mapsto 
\left[
\begin{array}{ccc}
1  & 1  & 1  \\
0  & 1  &0   \\
  0&  0 &  1 
\end{array}
\right] x + 
\left[
\begin{array}{c}
1   \\
0    \\
0 
\end{array}
\right]$
is self dual, but not linear, so not in $AP$.

\item $AD$: The map 
$x \mapsto 
\left[
\begin{array}{ccc}
1  & 1  & 1  \\
0  & 1  &0   \\
  0&  0 &  1 
\end{array}
\right] x 
$
is self dual,  linear, and 1-fixing, but not in $\Pi$.

\item $U=Deg(A)$: The map $NOT$ is unary, affine and non trivial, so not in $\Pi$.

\item $P_0$: The map 
\begin{equation*}
\begin{array}{|c|c|}
\hline
000 &  000  \\
001 &  001   \\
010 &  010 \\
011 & 111 \\
100 & 100 \\
101 & 101 \\
110 & 110\\
111 &   011\\
\hline
\end{array}
\end{equation*}
is  0-fixing , but not affine or 1-fixing.

\item $P_1$: The map 
\begin{equation*}
\begin{array}{|c|c|}
\hline
000 &  011  \\
001 &  001   \\
010 &  010 \\
011 & 111 \\
100 & 100 \\
101 & 101 \\
110 & 110\\
111 &   111\\
\hline
\end{array}
\end{equation*}
is  1-fixing , but not affine or 0-fixing.

\item $P$: The map 
\begin{equation*}
\begin{array}{|c|c|}
\hline
000 &  000  \\
001 &  010   \\
010 &  011 \\
011 &  001 \\
100 & 100 \\
101 & 101 \\
110 & 110\\
111 &  111\\
\hline
\end{array}
\end{equation*}
is  0-fixing and1-fixing , but not affine or self dual.

\end{itemize}
\qed
\end{proof}

Note that there are non-projection balanced monotone maps, e.g.
000,100,010,001 mapto 0, others map to 1. But this map cannot be extended to a permutation on $A^3$.

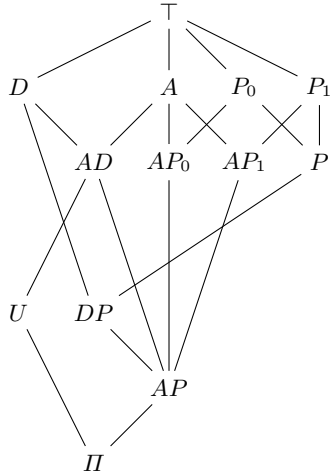
\begin{figure}[htbp]
\begin{center}
\begin{tikzpicture}
    \node (top) at (2,9) {$\top$};
    
    \node  (A) at (2,8) {$A$};
    \node (AP1) at (3,7){$AP_1$};
    \node [below  of=A] (AP0){$AP_0$};
    \node (AD) at (1,7){$AD$};
    \node (AP) at (2,4){$AP$};
    
    \node  (P0) at (3,8) {$P_0$};
    \node  (P1) at (4,8) {$P_1$};
    \node  (P) at (4,7) {$P$};

    \node  (D) at (0,8) {$D$};
    \node  (DP) at (1,5) {$DP$};

    \node (UD) at (0,5)   {$U$};
    
    \node (bot) at (1,3) {$\Pi$};
    
    
    \draw [] (top) -- (A);
    \draw [] (top) -- (D);
    \draw [] (top) -- (P0);
    \draw [] (top) -- (P1);

    \draw [] (P1) -- (P);
    \draw [] (P0) -- (P);
    \draw [] (P) -- (DP);
    \draw [] (P1) -- (AP1);
    \draw [] (P0) -- (AP0);
    \draw [] (DP) -- (AP);
    
    \draw [] (A) -- (AP0);
    \draw [] (A) -- (AP1);
    \draw [] (A) -- (AD);
    \draw [] (AD) -- (AP);
    \draw [] (AP0) -- (AP);
    \draw [] (AP1) -- (AP);
    \draw [] (AP) -- (bot);

    \draw [] (AD) -- (UD);
    \draw [] (UD) -- (bot);

    \draw [] (D) -- (DP);
    \draw [] (D) -- (AD);

\end{tikzpicture}
\caption{The inclusion diagram of the  clones on two elements  that give distinct permutation clones. This is a sublattice of the Post lattice.}
\label{figure_lattice_binary}
\end{center}
\end{figure}

From  Table \ref{table_binary}  we see that only two nontrivial permutation clones
on two elements are ancilla closed by Theorem \ref{thm_ancillaclosed}, 
the affine maps and the degenerate maps. 
And in fact these are the only relationally defined ancilla closed permutation clones on the binary alphabet according to \cite{aaronsonetal15}.

\section{Conclusions and further work}

The collection of relationally defined permutation clones is restricted enough to be comprehensible, 
but still contains many  interesting permutation clones.
These permutation clones are those that are defined by a property of their components and nothing else other than reversibility.
We have found the cluster definition for balanced maps.

We find that for a binary alphabet, there are only finitely many relationally defined 
permutation clones.
We cannot expect the general case to be so clean.
We observed that many relations $R$  give rise to trivial permutation clones.
Determining what relational properties imply triviality would be of value.
Can we  recognise a relation  that gives $\Pol(w_R)=\Pi$?


\subsubsection{Acknowledgements} Supported by the Austrian Science Fund (FWF): P33878  and PEEK Project AR561.
%
%
%
 \bibliographystyle{splncs04}
 \bibliography{RelPermClones}

\end{document}